\documentclass[11pt]{article}
\usepackage[utf8]{inputenc}

\usepackage{mathtools}
\usepackage{amsthm}
\usepackage{amsmath,amsfonts,amssymb,latexsym,amsthm,dsfont,amsbsy,stmaryrd}
\usepackage{caption}
\usepackage[left=3cm,right=3cm,top=1.8cm,bottom=2.3cm]{geometry}

\usepackage{tikz}
\usetikzlibrary{shapes,positioning}
\usepackage{xcolor,graphicx}
\usepackage{enumitem}

\usepackage[colorlinks=True]{hyperref}

\newtheorem{thm}{Theorem}

\newtheorem{prop}[thm]{Proposition}

\newtheorem{lem}[thm]{Lemma}

\newtheorem{rk}[thm]{Remark}

\renewcommand{\leq}{\leqslant}
\renewcommand{\geq}{\geqslant}

\newcommand{\dd}{{\mathrm d}}
\newcommand{\ee}{{\mathrm e}}

\newcommand{\dE}{\mathbb{E}}

\newcommand{\dP}{\mathbb{P}}
\newcommand{\dR}{\mathbb{R}}

\newcommand{\cF}{\mathcal{F}}

\newcommand{\cU}{\mathcal{U}}

\newcommand{\BRA}[1]{{{\left\{#1\right\}}}} 
\newcommand{\PAR}[1]{{{\left(#1\right)}}} 
\newcommand{\SBRA}[1]{{{\left[#1\right]}}} 

\title{Elephant polynomials}
\author{Hélène Guérin\thanks{Université du Québec à Montréal (UQAM), Département de Mathématiques, Canada. Funded by Natural Sciences and Engineering Research Council of Canada (NSERC) discovery grant (RGPIN-2020-07239); 
\texttt{guerin.helene@uqam.ca}},\ Lucile Laulin\thanks{Université Paris Nanterre, Modal’X, UMR CNRS 9023, UPL, France, and FP2M, CNRS FR 2036; \texttt{lucile.laulin@math.cnrs.fr}}\ and Kilian Raschel\thanks{Université d'Angers, CNRS, Laboratoire Angevin de Recherche en Mathématiques, France. This project has received funding from the European Research Council (ERC) under the European Union's Horizon 2020 research and innovation programme under the Grant Agreement No. 759702 and from Centre Henri Lebesgue,
programme ANR-11-LABX-0020-01; \texttt{raschel@math.cnrs.fr}}}
\date{\today}

\begin{document}

\maketitle
\begin{abstract}
In this note, we study a family of polynomials that appear naturally when analysing the characteristic functions of the one-dimensional elephant random walk. These polynomials depend on a memory parameter $p$ attached to the model. For certain values of $p$, these polynomials specialise to classical polynomials, such as the Chebychev polynomials in the simplest case, or generating polynomials of various combinatorial triangular arrays (e.g.\ Eulerian numbers). Although these polynomials are generically non-orthogonal (except for $p=\frac{1}{2}$ and $p=1$), they have interlacing roots. Finally, we relate some algebraic properties of these polynomials to the probabilistic behaviour of the elephant random walk. Our methods are reminiscent of classical orthogonal polynomial theory and are elementary.
\end{abstract}

\begin{quote}

\end{quote}

\section{Introduction and main results}

\paragraph{A one-parameter family of polynomials.}
In this paper, our main objective is to study a family of polynomials defined as follows: $R_1(x)=x$ and for $n\geq 1$,
\begin{equation}
\label{eq:recurrence_relation_Rn}
    R_{n+1}(x) = xR_n(x)-\frac{a}{n}(1-x^2)R_n'(x),
\end{equation}
where $a\in\mathbb R$ is some parameter. 
Due to a strong connection with the elephant random walk (ERW)\ when $a\in[-1,1]$, which we shall now present, we call them \emph{elephant polynomials}. The first three elephant polynomials are given by 
\begin{equation}
\label{eq:first_values_polynomials}
    R_1(x)  =x,\quad 
    R_2(x)  =(a+1)x^2 - a\quad  \text{and}\quad
    R_3(x)  =x((a+1)^2x^2- a(a+2)).
\end{equation}

\paragraph{The elephant random walk.}
The one-dimensional elephant random walk $(S_n)_{n\geq 0}$ is defined as follows \cite{Schutz2004}. We denote by $(X_n)_{n\geq 0}$ its successive steps. The elephant starts at the origin at time zero: $S_0 = 0$. For the first step $X_1$, the elephant moves one step to the right ($+1$) with probability $q$ or one step to the left ($-1$) with probability $1-q$, for some $q$ in $[0,1]$. The next steps are performed by choosing uniformly at random an integer $k$ among the previous times. Then the elephant moves exactly in the same direction as at time $k$ with probability $p\in[0,1]$, or in the opposite direction with probability $1-p$. In other words, defining for all $n \geq 1$, 
\begin{equation*}
   X_{n+1} = \left \{ \begin{array}{lll}
    +X_{k} &\text{with probability  $p$,} \\[1em]
    -X_{k} &\text{with probability $1-p$,}
   \end{array} \right.
\end{equation*}
with $k\sim\cU\BRA{1,\ldots,n}$, the position of the ERW at time $n+1$ is given by $S_{n+1}=S_{n}+X_{n+1}$. The probability $q$ is called the first step parameter (in this paper we will take $q=\frac{1}{2}$) and $p$ the memory parameter of the ERW.

\paragraph{The characteristic functions as trigonometric polynomials.}
The characteristic function of the process at time $n$ is defined by
\begin{equation}
\label{eq:def_charac}
    \varphi_n(t) = \mathbb E\SBRA{\ee^{i t S_n}}.
\end{equation}
Taking $q=\frac{1}{2}$, we have $\varphi_1(t) = \cos t$ and will justify later on that for $n\geq 1$, 
\begin{equation}
\label{eq:recurrence_varphi}
    \varphi_{n+1}(t) = \cos(t)\varphi_n(t)+\frac{a}{n}\sin(t)\varphi'_n(t),
\end{equation}
with $a=2p-1\in[-1,1]$. The sequence $\{\varphi_n(t)\}_{n\geq 1}$ naturally defines a sequence of polynomials $\{R_n(x)\}_{n\geq 1}$ via the formula
\begin{equation}
\label{eq:def_R_n}
   \varphi_n(t) = R_n(\cos t).
\end{equation}
Using \eqref{eq:recurrence_varphi}, one immediately obtain the recurrence relation \eqref{eq:recurrence_relation_Rn} mentioned at the beginning of the paper. Although the elephant random walk model is classical and well studied in the literature (see e.g.\ \cite{BaurBertoin2016,BercuLaulin2019,Bertoinb2021,Coletti2017,Coletti2019,Cressoni2007,daSilva2013,GLR2023,Kursten2016,qin2023recurrence,Schutz2004}), the properties of the polynomials we will prove in this paper have remained unnoticed.

Interestingly, specialising the parameter $a$ to some particular values, we recover classical orthogonal or combinatorial polynomials. Our main results in this direction can be summarized in Table~\ref{tab:recap}. 
The distribution of $S_n$ in these specific cases is consistent with \cite{Bertoinb2021} on the number of returns to zero, and the fact that the random walk is positive recurrent for $a<\frac{1}{2}$. 
\begin{table}[ht!]
\begin{center}
   \begin{tabular}{ | c | c | c |c |c | c |}
   \hline 
   $a$ & $-1$ & $-\frac{1}{2}$ & $0$ & $1$ & $\pm\infty$\\[0.1cm]
     \hline
     $R_n(x)$ & \begin{minipage}{0.14\textwidth}\begin{center}\smallskip\href{https://oeis.org/A101280}{A101280}

     (Prop.~\ref{prop:a=-1})
    \end{center}
     \end{minipage}& 
     \begin{minipage}{0.14\textwidth}
     \begin{center}
     \smallskip\href{https://oeis.org/A034839}{A034839}

     (Prop.~\ref{prop:a=-1/2})
    \end{center}
     \end{minipage} & \begin{minipage}{0.14\textwidth}\begin{center}\smallskip $x^n$

     (Lem.~\ref{lem:obvious_properties_EP})
    \end{center}
     \end{minipage}  & 
     \begin{minipage}{0.14\textwidth}\begin{center}\smallskip Chebychev

     (Lem.~\ref{lem:obvious_properties_EP})
    \end{center}
     \end{minipage}  &  
     \begin{minipage}{0.15\textwidth}\begin{center}\smallskip\href{https://oeis.org/A008293}{A008293}

     (Prop.~\ref{prop:a_infty})
    \end{center}
     \end{minipage}\\[0.4cm] \hline
      $\mathbb P(S_n=k)$ & 
       \begin{minipage}{0.14\textwidth}
     \begin{center}
     \smallskip $\frac{A(n-1,\frac{n+k}{2})}{(n-1)!}$

     (Prop.~\ref{prop:p=0_dist}) \cite{Mahmoud2008,Bertoin2023}
    \end{center}
     \end{minipage} & 
     \begin{minipage}{0.14\textwidth}\vspace{0.8mm}
     \begin{center}
     \smallskip $\frac{1}{2^{2n-1}}\binom{2n}{n+k}$
\smallskip

     (Prop.~\ref{prop:p=1/4_dist})
    \end{center}
     \end{minipage} 
      & 
      \begin{minipage}{0.14\textwidth}\vspace{0.8mm}
     \begin{center}
     \smallskip $\frac{1}{2^{n}}\binom{n}{\frac{n+k}{2}}$

     (simple RW)
    \end{center}
     \end{minipage}
      & $\frac{\delta_{k,+n}+\delta_{k,-n}}{2}$  & n.a.\\[.5cm]
     \hline
   \end{tabular}
 \end{center}
\caption{For certain values of $a$, the polynomials $R_n(x)$ specialise to classical polynomials or classical combinatorial sequences, and the exact distribution of the random walk is computed.
In the last row, the quantities $A(n,k)$ stand for the Eulerian numbers, see \href{https://oeis.org/A008292}{A008292} in \cite{oeis}; by definition, they count 
increasing rooted trees with $n+1$ nodes and $k$ leaves.}
\label{tab:recap}
\end{table}

The very special form of the polynomials $R_n$ at $a=0$ and $a=1$ in Table~\ref{tab:recap} admits a direct probabilistic interpretation. For $a=0$,  the ERW's memory parameter  $p=\frac{1+a}{2}=\frac{1}{2}$, and $\varphi_n(t)$ is the characteristic function of the classical simple random walk on $\mathbb Z$, namely $\cos(t)^n$. It is also clear that for $a=1$, one should get for $\varphi_n(t)$ the Chebychev polynomial of the first kind, namely, $\varphi_n(t) = \cos( nt)$, as from a probabilistic point of view the ERW's memory parameter $p=1$, this means that at time $n$ the elephant is either at $n$ or $-n$, with probability $\frac{1}{2}$ each. On the other hand, we don't have a probabilistic interpretation in the other cases appearing in Table~\ref{tab:recap}, for instance $a=-1$ ($p=0$) and $a=-\frac{1}{2}$ ($p=\frac{1}{4}$). Note that from a probabilistic viewpoint, it does not make sense to take $a>1$ or $a<-1$ (which corresponds to $p>1$ or $p<0$); however, the recursive definition \eqref{eq:recurrence_relation_Rn} is well defined for any value of $a$. 
 
In addition to the results presented in Table~\ref{tab:recap}, which only concern a few values of $a$, we will prove structural results on the roots of $R_n(x)$. Our first result analyses the situation $a>0$. 
\begin{prop}
\label{prop:interlacing}
    For $a>0$ and $n\geq 1$, $R_n$ admits $n$ real roots, which are mutually distinct and on $(-1,1)$. Moreover, the zeros of $R_n$ and the zeros of $R_{n+1}$ interlace.
\end{prop}
See Figure~\ref{fig:a>0} for an illustration of Proposition~\ref{prop:interlacing}. 
\begin{figure}[ht!]
    \centering
    \includegraphics[height=4cm]{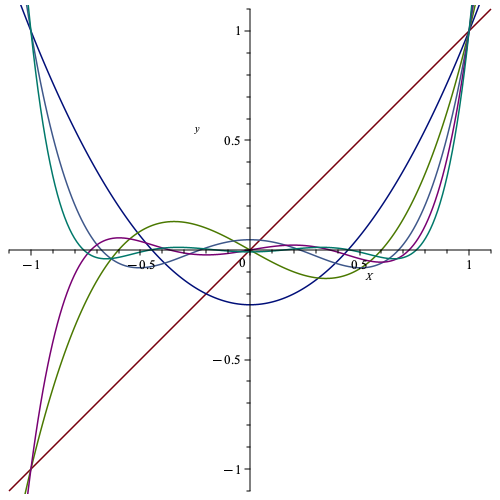}\qquad\qquad \includegraphics[height=4cm]{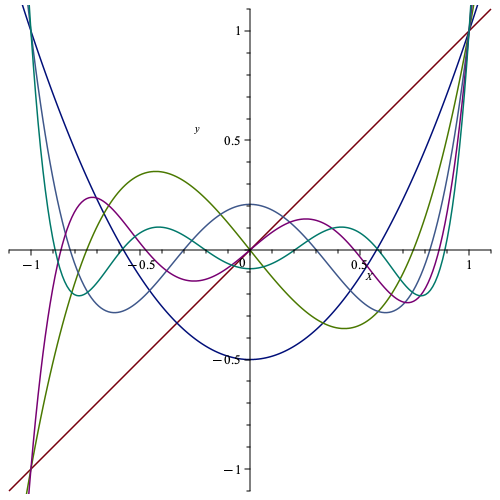}
    
    \includegraphics[height=4cm]{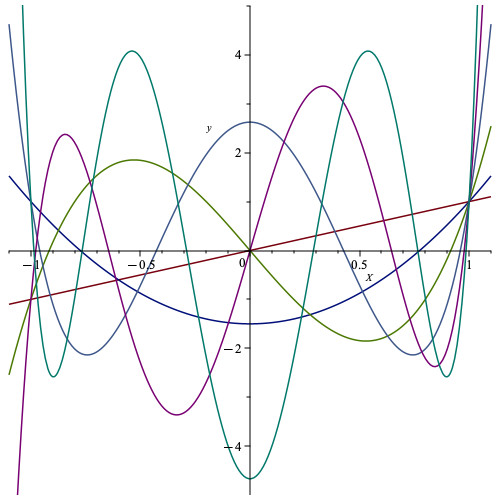}\qquad\qquad\includegraphics[height=4cm]{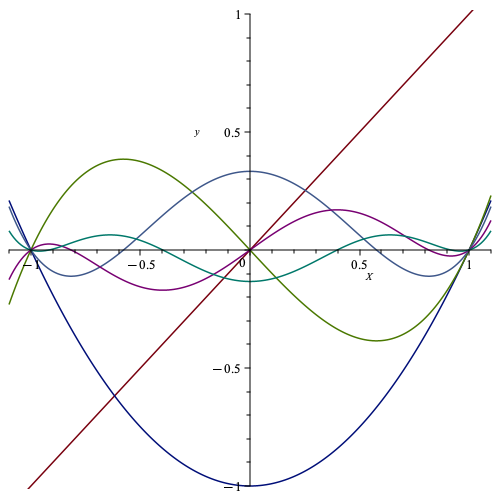}
    \caption{According to Proposition~\ref{prop:interlacing}, the first six polynomials $R_1(x),\ldots,R_6(x)$ have interlaced roots. Top left display: $a=\frac{1}{4}$; top right: $a=\frac{1}{2}$, these are the Chebychev polynomials; bottom left: $a=\frac{3}{2}$; bottom right: $a=+\infty$.}
    \label{fig:a>0}
\end{figure}
When $a=0$ we have $R_n(x)=x^n$ (see \eqref{eq:recurrence_relation_Rn}), so that all roots collapse at $0$. In the result below, we prove that the behavior of the roots dramatically changes when $a<0$, as they all become purely imaginary. However, they still possess an interlacing property.
\begin{prop}
\label{prop:interlacing_imaginary}
    Define $S_n(x) = (-i)^n R_n(ix)$. For $a<0$ and $n\geq 1$, $S_n(x)$ is a real polynomial, which admits $n$ real roots if $a\neq -1$, and $\vert n-2\vert $ real roots if $a=-1$, which are mutually distinct in all cases. Moreover, the zeros of $S_n$ and the zeros of $S_{n+1}$ interlace.
\end{prop}
See Figure~\ref{fig:a<0} for an illustration of Proposition~\ref{prop:interlacing_imaginary}.
\begin{figure}
    \centering
    \includegraphics[height=4cm]{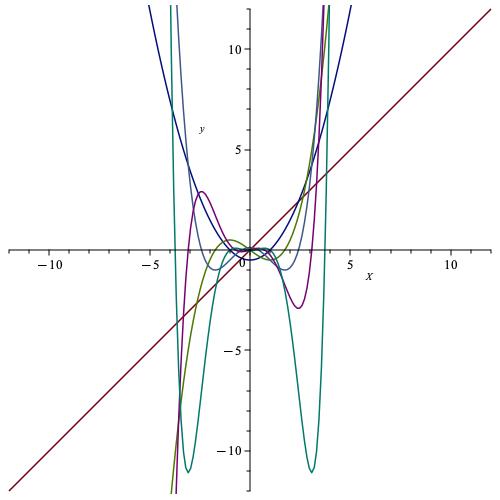}\qquad\qquad\includegraphics[height=4cm]{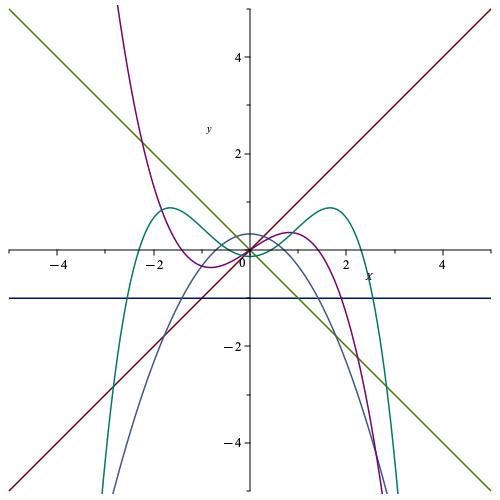}
    
    \includegraphics[height=4cm]{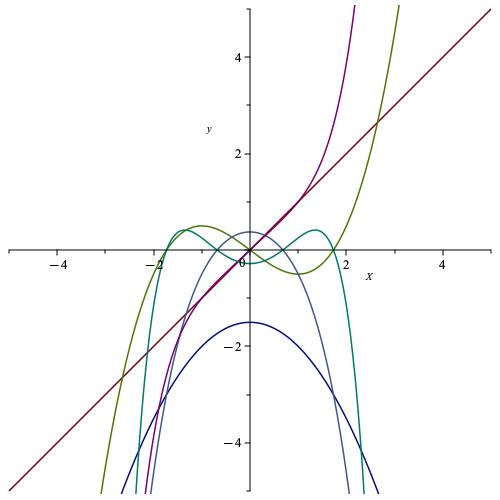}\qquad\qquad\includegraphics[height=4cm]{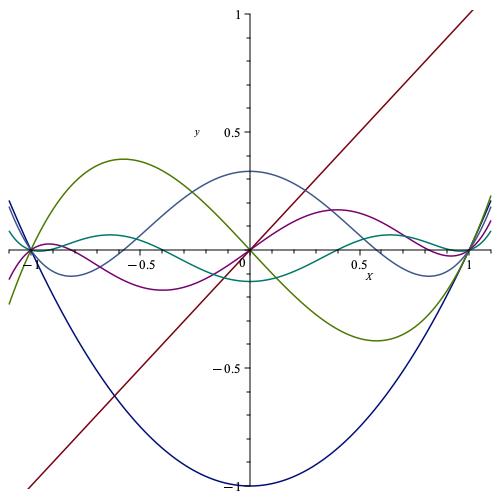}
    \caption{From top left  to bottom right, the first six polynomials $S_1(x),\ldots,S_6(x)$ for $a=-\frac{1}{2},-1,-\frac{3}{2},-\infty$  have interlaced roots, according to Proposition~\ref{prop:interlacing_imaginary} (except $S_2(x)=-1$ in the case $a=-1$)}
    \label{fig:a<0}
\end{figure}

The interlacing property is most classical for orthogonal polynomials \cite{zbMATH02510415}, so it is useful to notice that:
\begin{prop}
\label{prop:non_orthogonal}
Except for $a=0$ and $a=1$, the $R_n(x)$ and $S_n(x)$ are not orthogonal.
\end{prop}

Finally, from a probabilistic point of view, it is natural to compute the asymptotics of $R_n'(1)$ as $n\to\infty$, in order to guess limit theorems for the elephant random walk. Indeed, if $a=0$, the model reduces to the classical simple random walk, for which the most classical central limit theorem yields
\begin{equation*}
    \varphi_n\Bigl(\frac{t}{\sqrt n}\Bigr)=R_n\Bigl(\cos \Bigl(\frac{t}{\sqrt n}\Bigr)\Bigr)\underset{n\to\infty}{\longrightarrow} \ee^{-\frac{t^2}{2}}.
\end{equation*}
When $n\to\infty$, 
\begin{equation*}
   R_n\Bigl(\cos \Bigl(\frac{t}{\sqrt n}\Bigr)\Bigr)=R_n\Bigl(1-\frac{t^2}{2n}+O\Bigl(\frac{1}{n^2}\Bigr)\Bigr)=1-R_n'(1)\frac{t^2}{2n}+(R_n'(1) + 3R_n''(1))\frac{t^4}{24n^2}+\ldots, 
\end{equation*}
so the asymptotics of the derivatives of $R_n(x)$ at $x=1$ should reflect the correct scaling to have a central limit theorem.

\begin{prop}
\label{prop:values_Rn'(1)}
For all $n\geq 1$ and $a\in[0,1]$,
    \begin{equation}
    \label{eq:values_Rn'(1)}
    R_n'(1) = \frac{\Gamma(n + 2)\Gamma(2a) - \Gamma( n + 1+2a )}{(1 - 2a)\Gamma(2a)\Gamma(n + 1)}.
\end{equation}
Moreover, we have the following asymptotics as $n\to\infty$:
\begin{equation*}
    R_n'(1) \sim \left\{\begin{array}{ll}
      \displaystyle  \frac{n}{1-2a} & \text{if } a< \frac{1}{2},\medskip\\
      \displaystyle  n\ln(n)& \text{if } a= \frac{1}{2},\medskip\\
      \displaystyle \frac{n^{2a}}{(2a-1)\Gamma(2a)} & \text{if } a> \frac{1}{2}.
      \end{array}\right.
\end{equation*}
\end{prop}
As $R_n'(1)$ represents the second-order moment of the elephant random walk, the proposition above aligns with the initial (and by now well-established) observations regarding the asymptotic behavior of this moment, such as in \cite{Schutz2004}.
The values of the higher-order derivatives $R_n''(1)$, etc., seems to be fairly more complicated.

\section{Interlacing property}

We start with stating elementary properties of the $R_n(x)$.
\begin{lem}
\label{lem:obvious_properties_EP}
    For all $a\in\mathbb R$, $R_n$ is odd (resp.\ even) for odd (resp.\ even) values of $n$, and
    one has $R_n(1)=1$ and $R_n(-1)=(-1)^n$. For $a\in \mathbb R\setminus \{-1\}$, the polynomial $R_n$ has degree $n$ and dominant coefficient $(a+1)^{n-1}$. Moreover, for $a=0$ one has $R_n(x)=x^n$ and for $a=1$, $R_n(x)=T_n(x)$, the $n$th Chebychev polynomial of the first kind. For $a=-1$, one has $R_1(x)=x$ and for $n\geq 2$, $R_n(x)$ has degree $ n-2$, with dominant coefficient $\frac{2^{n-2}}{(n-1)!}$.
  Furthermore, when $a>0$ the coefficients of the polynomial $R_n$ have alternating signs, and when $a\in[-1,0]$ the coefficients are nonnegative. Viewed as a polynomial in $a$, $R_n$ has degree $n-1$.
\end{lem}

\begin{proof}
    The above lemma is obvious, with perhaps the exception of the connection with Chebychev polynomials. By definition $T_n(\cos t) = \cos (nt)$, and then it is clear that $\cos (nt)$ satisfies \eqref{eq:recurrence_varphi} as well as the initial condition, when $a=1$. The sign of the coefficients is easily obtained by induction.
\end{proof}

\begin{proof}[Proof of the identity~\eqref{eq:recurrence_varphi}]
Remember that $\varphi_n(t)$ is the characteristic function of the position $S_n$ of the elephant random walk at time $n$, see \eqref{eq:def_charac}. Let us denote by $\PAR{\cF_n}_{n\geq 0}$ the natural filtration of $\PAR{S_n}_{n\geq 0}$. It is known that by simple calculations we have  $\dP\PAR{X_{n+1}=1\vert\cF_n}=\frac{1}{1}\PAR{1+a\frac{S_n}{n}}$. Then,
\begin{align*}
    \varphi_{n+1}(t)&=\dE\SBRA{\ee^{itS_{n+1}}}\\
    &=\dE\SBRA{\ee^{itS_n}\dE\SBRA{\ee^{itX_{n+1}}\vert \cF_n}} \\
&=\dE\SBRA{\ee^{itS_n}\PAR{\frac{1}{2}\PAR{1+a\frac{S_n}{n}}\ee^{it}+\frac{1}{2}\PAR{1-a\frac{S_n}{n}}\ee^{-it}}}\\
&=\frac{\ee^{it}+\ee^{-it}}{2}\dE\SBRA{\ee^{itS_n}}+a\frac{\ee^{it}-\ee^{-it}}{2n}\dE\SBRA{S_n\ee^{itS_n}}\\
&=\cos(t)\varphi_n(t)+\frac{a}{n}\sin(t)\varphi'_n(t).\qedhere
\end{align*}
\end{proof}

\begin{proof}[Proof of Proposition~\ref{prop:interlacing}]
The techniques used are elementary and reminiscent of orthogonal polynomials theory, see \cite{zbMATH02510415}.
Let us denote the $n$ (a priori complex and non-distinct) roots of $R_n$ by $\alpha_1^{(n)},\ldots,\alpha_n^{(n)}$. We will prove by induction that for all $n\geq 1$, the $\alpha_k^{(n)}$ are real, are in $(-1,1)$ and satisfy $\text{sign}\bigl(R_n'(\alpha_k^{(n)})\bigr)=(-1)^{k+n}$. In particular, they should be simple roots.

This is easily verified for $n=1$, with $R_1(x) = x$ and $\alpha_1^{(1)}=0$. Let us now assume that the previous assumption holds for some $n$. Using \eqref{eq:recurrence_relation_Rn}, one immediately obtain that
\begin{equation*}
    R_{n+1}\bigl(\alpha_k^{(n)}\bigr)=-\frac{a}{n} \bigl(1-\bigl(\alpha_k^{(n)}\bigr)^2\bigr)R_n'\bigl(\alpha_k^{(n)}\bigr).
\end{equation*}
In particular, since $a>0$ we deduce that 
\begin{equation}
    \label{eq:signs_intermediate}
    \text{sign}\bigl(R_{n+1}(\alpha_k^{(n)})\bigr)=(-1)^{k+n+1}.
\end{equation}
Introduce $\alpha_0^{(n)}=-1$ and $\alpha_{n+1}^{(n)}=1$. Note that \eqref{eq:signs_intermediate} is true for $k=0$ and $k=n+1$ as well, using $R_{n+1}(1)=1$ and $R_{n+1}(-1)=(-1)^{n+1}$ (see Lemma~\ref{lem:obvious_properties_EP}). We now apply to the function $R_{n+1}$ the intermediate value theorem on the interval $[\alpha_k^{(n)},\alpha_{k+1}^{(n)}]$, for any $k$ from $0$ to $n$. Since by \eqref{eq:signs_intermediate} the signs of $R_{n+1}(\alpha_k^{(n)})$ and $R_{n+1}(\alpha_{k+1}^{(n)})$ are opposite, we immediately obtain the existence of a zero, which we denote by $\alpha_{k+1}^{(n+1)}$. By construction we obtain $n+1$ points in $(-1,1)$, which are mutually distinct. Finally, since these zeros must be simple, the derivatives $R_{n+1}'(\alpha_k^{(n+1)})$ must be non-zero and of alternating sign.
\end{proof}

\begin{proof}[Proof of Proposition~\ref{prop:interlacing_imaginary}]
It is very similar to the one of Proposition~\ref{prop:interlacing}. Let us first assume that $a\in(-1,0)$. Using \eqref{eq:recurrence_relation_Rn} we obtain the recurrence relation
\begin{equation}
\label{eq:recurrence_relation_Sn}
    S_{n+1}(x) = xS_n(x) +\frac{a}{n}(1+x^2)S_n'(x),
\end{equation}
which is valid for all $n\geq1$. The dominant coefficient of $S_n(x)$ is $(a+1)^{n-1}x^n$ by Lemma~\ref{lem:obvious_properties_EP}, so that 
\begin{equation}
    \label{eq:values_Sn_infty}
    S_n(+\infty)=+\infty\quad \text{and}\quad S_n(-\infty)=(-1)^n\infty.
\end{equation} 
Following the proof of the previous proposition, let us denote the $n$ (a priori complex and non-distinct) roots of $S_n$ by $\beta_1^{(n)},\ldots,\beta_n^{(n)}$. We will prove by induction that for all $n\geq 1$, the $\beta_k^{(n)}$ are real and satisfy $\text{sign}\bigl(S_n'(\beta_k^{(n)})\bigr)=(-1)^{k+n}$. In particular, they should be simple roots. Using \eqref{eq:recurrence_relation_Sn} we find
\begin{equation*}
    S_{n+1}\bigl(\beta_k^{(n)}\bigr)=\frac{a}{n} \bigl(1+\bigl(\beta_k^{(n)}\bigr)^2\bigr)S_n'\bigl(\beta_k^{(n)}\bigr).
\end{equation*}
In particular, since $a<0$ we deduce that 
\begin{equation}
    \label{eq:signs_intermediate_S}
    \text{sign}\bigl(S_{n+1}(\beta_k^{(n)})\bigr)=(-1)^{k+n+1}.
\end{equation}
Introduce $\beta_0^{(n)}=-\infty$ and $\beta_{n+1}^{(n)}=+\infty$. Note that \eqref{eq:signs_intermediate_S} is true for $k=0$ and $k=n+1$ as well, using the limits $S_{n+1}(\pm\infty)$ computed in \eqref{eq:values_Sn_infty}. We now apply to the function $S_{n+1}$ the intermediate value theorem on the interval $[\beta_k^{(n)},\beta_{k+1}^{(n)}]$, for any $k$ from $0$ to $n$. Since by \eqref{eq:signs_intermediate_S} the signs of $S_{n+1}(\beta_{k}^{(n)})$ and $S_{n+1}(\beta_{k+1}^{(n)})$ are opposite, we immediately obtain the existence of a zero, which we denote by $\beta_{k+1}^{(n+1)}$. By construction we obtain $n+1$ points in $\dR$, which are mutually distinct. Finally, since these zeros must be simple, the derivatives $S_{n+1}'(\beta_k^{(n+1)})$ must be non-zero and of alternating sign.

If now $a<-1$, we prove similarly by induction that for all $n\geq 1$, the $\beta_k^{(n)}$ are real and satisfy $\text{sign}\bigl(S_n'(\beta_k^{(n)})\bigr)=(-1)^{k}$, using the following changes: \eqref{eq:values_Sn_infty} should be replaced by $S_n(+\infty)=(-1)^{n-1}\infty$ and $S_n(-\infty)=-\infty$; Equation~\eqref{eq:signs_intermediate_S} should be modified as follows: $\text{sign}\bigl(S_{n+1}(\beta_k^{(n)})\bigr)=(-1)^{k+1}$.

Finally, the case $a=-1$ is very similar to the previous situation where $a<-1$, with the only difference that the degree of $S_n$ is $n-2$. More precisely, one has $S_1(x) = x$, $S_2(x)=-1$ and for $n\geq 2$, the dominant term of $S_{n}(x)$ is $-\frac{2^{n-2}}{(n-1)!}x^{n-2}$. The proof continues similarly as in the case $a<-1$.
\end{proof}

\begin{proof}[Proof of Proposition~\ref{prop:non_orthogonal}]
Let us do the proof in the case of $R_n(x)$; one would conclude in the case of the sequence $S_n(x)$ by similar arguments. If the sequence $R_n(x)$ is orthogonal, it has to satisfy a recurrence relation of the form
\begin{equation}
\label{eq:guessed_recurrence_poly}
    R_{n+1}(x) = (\alpha_nx+\beta_n)R_n(x) + \gamma_nR_{n-1}(x).
\end{equation}
Analysing the dominant coefficient (see Lemma~\ref{lem:obvious_properties_EP}), one should have $\alpha_n=a+1$. Moreover, due to the parity of the $R_n(x)$, see again Lemma~\ref{lem:obvious_properties_EP}, one must take $\beta_n=0$. We thus deduce $\gamma_n=-a$ from $R_n(1)=1$ (see Lemma~\ref{lem:obvious_properties_EP}),  so that \eqref{eq:guessed_recurrence_poly} becomes
\begin{equation}
\label{eq:guessed_recurrence_poly_bis}
    R_{n+1}(x) = (a+1)x R_n(x) -a R_{n-1}(x).
\end{equation}
Actually, setting $R_0(x)=1$, the identity \eqref{eq:guessed_recurrence_poly_bis}
holds true for $n=1$ and $2$; compare with \eqref{eq:first_values_polynomials}. 

We now look at \eqref{eq:guessed_recurrence_poly_bis} in the case $n=3$. Computing $R_4(x)-(a+1)xR_3(x)$ gives a second-degree polynomial, which if \eqref{eq:guessed_recurrence_poly_bis} were true, should be proportional to $R_2(x)$. However, an elementary computation shows that the resultant of the polynomials $R_4(x)-(a+1)xR_3(x)$ and $R_2(x)$ is given by $\frac{a^4(a - 1)^2}{9}$. In other words, except for $a=0$ and $a=1$, the recurrence relation \eqref{eq:guessed_recurrence_poly_bis} (and thus \eqref{eq:guessed_recurrence_poly} as well)\ is not satisfied for $n=3$.
\end{proof}

\section{Special cases of the memory parameter}

\subsection{The case \texorpdfstring{$a=-\frac{1}{2}$}{of a memory parameter one fourth}}
\begin{prop}
\label{prop:a=-1/2}
Let $a=-\frac{1}{2}$. For all $n\geq 1$, we have
\begin{equation*}
    R_n(x) = \frac{1}{2^{n-1}}\sum_{k=0}^{\lfloor \frac{n}{2}\rfloor} \binom{n}{2k}x^{n-2k}=\left(\frac{x-1}{2}\right)^n+\left(\frac{x+1}{2}\right)^n.
\end{equation*}
\end{prop}
And indeed, for $a=-1/2$, with \eqref{eq:recurrence_varphi} and \eqref{eq:def_R_n} we find
\begin{equation*}
    R_1(x) = x,\quad
    R_2(x) = \frac{x^2+1}{2},
    \quad R_3(x) =\frac{x^3+3x}{2^2} ,\quad 
    R_4(x) = \frac{x^4+6x^2+1}{2^3},\quad \text{etc.}
\end{equation*}
The coefficients in the expansions of the above polynomials are thus connected with the sequence \href{https://oeis.org/A034839}{A034839}, corresponding to the triangular array formed by taking every other term of each row of Pascal's triangle.

\begin{proof}[Proof of Proposition~\ref{prop:a=-1/2}]
It is a direct consequence of the recurrence equation \eqref{eq:recurrence_relation_Rn}. Indeed, an obvious computation shows that 
\begin{multline*}
    \left(\frac{x-1}{2}\right)^{n+1}+\left(\frac{x+1}{2}\right)^{n+1}=x\left(\left(\frac{x-1}{2}\right)^n+\left(\frac{x+1}{2}\right)^n\right)\\+\frac{1-x^2}{4}\left(\left(\frac{x-1}{2}\right)^{n-1}+\left(\frac{x+1}{2}\right)^{n-1}\right),
\end{multline*}
which corresponds to \eqref{eq:recurrence_relation_Rn} with $a=-\frac{1}{2}$. Therefore $R_n(x)$ and $\left(\frac{x-1}{2}\right)^n+\left(\frac{x+1}{2}\right)^n$ have the same initial term for $n=1$ and satisfy the same recurrence, hence are equal for all $n$.
\end{proof}

\begin{prop}
\label{prop:p=1/4_dist}
Let $a=-\frac{1}{2}$ (equivalently  $p=\frac{1}{4}$). For all $n\geq 1$, $-n\leq k\leq n$, $k$ and $n$ having same parity, one has
\begin{equation}
\label{eq:p=1/4_dist}
    \mathbb P(S_{n}=k) = \frac{1}{2^{2n-1}}\binom{2n}{n+k}.
\end{equation}
In particular, for any $n\geq 1$, we have
    \begin{equation*}
        \mathbb P(S_{2n}=0) = \frac{1}{2^{4n-1}}\binom{4n}{2n}\underset{n \to\infty}{\sim} \sqrt{\frac{2}{\pi n}}.
    \end{equation*}
\end{prop}

It is interesting to compare \eqref{eq:p=1/4_dist} with the distribution of the classical symmetric simple random walk on $\mathbb Z$ (corresponding to $p=\frac{1}{2}$ in our model): $\mathbb P(S_{n}=k) = \frac{1}{2^{n}}\binom{n}{\frac{n+k}{2}}$, it is as if all elements of the binomial coefficient should be divided by $2$.

\begin{proof}
For the sake of conciseness, we focus on the case $k=0$. Due to the periodicity of the model and for any memory parameter, one has $\mathbb P(S_{2n-1}=0)=0$. Moreover,
    by Proposition~\ref{prop:a=-1/2} we have
    \begin{equation}
    \label{eq:btGF}
        \mathbb P(S_{2n}=0) =\frac{1}{2\pi} \int_{-\pi}^\pi \varphi_{2n}(t) dt= \frac{1}{2\pi} \int_{-\pi}^\pi \left(\Bigl(\frac{\cos t-1}{2}\Bigr)^{2n}+\Bigl(\frac{\cos t+1}{2}\Bigr)^{2n}\right) \dd t,
    \end{equation}
from which the value given in Proposition~\ref{prop:p=1/4_dist} follows from standard integral computations (or simply taking generating functions in \eqref{eq:btGF}). We obtain similarly the complete distribution of $S_n$, using the relation
\begin{equation*}
   \dP(S_n=k)=\frac{1}{2\pi}\int_{-\pi}^{\pi}\ee^{-ikt}\varphi_n(t)\dd t.\qedhere
\end{equation*}
\end{proof}

\subsection{The case \texorpdfstring{$a=-1$}{of a memory parameter zero}}

Define a triangle of integers $T(n,k)$ for $n\geq 1$ and $0\leq k\leq \lfloor\frac{n-1}{2}\rfloor $ as follows:  
\begin{equation}
    \label{eq:def_triangle}
    T(1,0)=1\quad \text{and for $n\geq 2$,}\quad  T(n,k) = (k+1)T(n-1,k) + (2n-4k)T(n-1,k-1).
\end{equation} 
Further, for all $n\geq 1$, define 
\begin{equation}
    \label{eq:def_poly_U}
    U_n(y ) = \sum_{k=0}^{\lfloor\frac{n-1}{2}\rfloor} T(n,k)y^k. 
\end{equation}
The $T(n,k)$ and $U_n(y)$ for $n$ up to $7$ are reproduced below:
\begin{equation*}
\left\{\begin{array}{lrcl}
    1,& U_1(y) &=& 1,\\
    1,& U_2(y) &=& 1,\\
    1, 2,& U_3(y) &=& 1+2y,\\
    1, 8,& U_4(y) &=& 1+8y,\\
    1, 22, 16,& U_5(y) &=& 1+22y+16y^2,\\
    1, 52, 136,& U_6(y) &=& 1+52y+136y^2,\\
    1, 114, 720, 272,\qquad &U_7(y) &=& 1+114y+720y^2+272y^3.
    \end{array}\right.
\end{equation*}
See \href{https://oeis.org/A101280}{A101280} in the OEIS for various properties and characterizations of these numbers.

\begin{prop}
\label{prop:a=-1}
Let $a=-1$  and $U_n$ be the polynomials defined in \eqref{eq:def_poly_U}. We have for all $n\geq 2$,
\begin{equation}
\label{eq:change_var}
    R_n(x) = \frac{(2x)^{n-2}}{(n-1)!}U_{n-1}\left(\frac{1}{4x^2}\right).
\end{equation}
\end{prop}

\begin{proof}
Let us make a change of function as in \eqref{eq:change_var}. A straightforward computation from \eqref{eq:recurrence_relation_Rn} (with $a=-1  $) shows that the polynomial sequence $(U_n)_{n\geq0}$ defined by \eqref{eq:change_var} must satisfy the following recurrence relation: $U_1(y)=1$ and for $n\geq 2$,
\begin{equation*}
    U_n(y) = (2ny+ 1-4y)U_{n-1}(y) +y(1-4y)U_{n-1}'(y).
\end{equation*}
On the other hand, based on \eqref{eq:def_triangle} it is clear that the polynomials defined in \eqref{eq:def_poly_U} satisfy the same recurrence equation, so we conclude by uniqueness.
\end{proof}

In the OEIS \cite{oeis}, it is mentioned that the generating function of the $U_n(y)$ admits the following closed-form expression:
\begin{equation}
\label{eq:formula_U_n_OEIS}
   \sum_{n\geq 1} U_n(y)\frac{z^n}{n!} =\sum_{n\geq 1, k\geq 0} T(n, k) y^k \frac{z^n}{n!} =  \frac{C(y)(2-C(y))}{\exp\bigl(-z \sqrt{1-4y}\bigr) + 1 - C(y)} - C(y),
\end{equation}
where $C(y) = \frac{1 - \sqrt{1-4y}}{2y}$ is the generating function of Catalan numbers. Notice that this equation holds for $(y,z)$ such that $\vert y\vert <\frac{1}{4}$ and $z\in\mathbb C$.

We now turn to the distribution of the random walk when $a=-1$, and give a new proof of the result. Denote by $A(n,k)$ the Eulerian numbers, defined for $n\geq 1$ and $k\in\{1,\ldots,n\}$. Besides their combinatorial interpretation in terms of increasing rooted trees with $n+1$ nodes and $k$ leaves, the Eulerian numbers admit the following probabilistic interpretation: for $k\in\{1,\ldots,n\}$, $\frac{A(n,k)}{n!}$ is the probability that a sum of $n$ independent uniform random variables on $[0,1]$ lies between $k-1$ and $k$, see \cite{tanny}.

\begin{prop}[\cite{Mahmoud2008,Bertoin2023}]
\label{prop:p=0_dist}
Let $a=-1$ (equivalently $p=0$). For any $n\geq 2$ and $-(n-2)\leq k\leq n-2$ with the same parity as $n$, we have
\begin{equation*}
        \mathbb P(S_{n}=k)=\frac{A\bigl(n-1,\frac{n+k}{2}\bigr)}{(n-1)!}.
    \end{equation*}  
In particular, for any $n\geq 1$, $\mathbb P(S_{2n-1}=0)=0$ and 
    \begin{equation*}
        \mathbb P(S_{2n}=0)=\frac{A(2n-1,n)}{(2n-1)!}\underset{n \to\infty}{\sim} \sqrt{\frac{3}{\pi n}},
    \end{equation*}
where $A(2n-1,n)$ denotes the \href{https://oeis.org/A025585}{central Eulerian number}, equal to $\sum_{k=0}^n (-1)^k(n-k)^{2n-1}\binom{2n}{k}$. 
\end{prop}
It suggests that a bijective proof of Proposition~\ref{prop:p=0_dist} should exist, meaning to find a bijection between ERW of length $n$, ending at altitude $k$, and  increasing rooted trees with $n+1$ nodes and $k$ leaves, with the parity condition on $n$ and $k$ mentioned in Proposition~\ref{prop:p=0_dist}. Using the connection between ERW and P\'olya urn models, let's mention that this result was already known  for $p=0$ (see \cite[Sec.~7.2.2]{Mahmoud2008} and \cite[Lem.~2.1]{Bertoin2023}). It is also interesting to note that the P\'olya urn model in this case is equivalent to a time-shifted Internal Diffusion Limited Aggregation  model (also called growth model visited by an explorer in \cite{Diaconis-Fulton91}), for which the distribution had also been obtained in \cite[Thm~1]{Mittelstaedt20}. In these earlier articles, other techniques were used to prove the result.

\begin{proof}[Proof of Proposition~\ref{prop:p=0_dist}]
We prove the statement about the return probability $\mathbb P(S_{2n}=0)$, the exact same proof would allow us to describe the full distribution of the model.
It follows from \eqref{eq:change_var} that
\begin{equation*}
        \mathbb P(S_{2n}=0) 
        =\frac{1}{2\pi} \int_{-\pi}^\pi \varphi_{2n}(t) \dd t 
        =\frac{1}{2\pi} \frac{2^{2n-2}}{(2n-1)!}\int_{-\pi}^\pi \bigl(\cos t\bigr)^{2n-2} U_{2n-1}\Bigl(\frac{1}{4\cos^2 t}\Bigr)\dd t.
\end{equation*}
We now use the following identity, stated in \href{https://oeis.org/A101280}{A101280}:
\begin{equation*}
    (1+x)^{2n-2} U_{2n-1}\left(\frac{x}{(1+x)^2}\right) = \sum_{k=0}^{2n-2} A(2n-1,k+1)x^k = A_{2n-1}(x),
\end{equation*}
where the last identity is the definition of the $(2n-1)$th Eulerian polynomial. Setting $x=\ee^{it}$, we deduce that
\begin{equation}
\label{eq:relation_U_A}
    \bigl(2\cos t\bigr)^{2n-2} U_{2n-1}\Bigl(\frac{1}{4\cos^2 t}\Bigr) = \exp\bigl(- 2i{(n-1)}t\bigr) A_{2n-1}\bigl(\exp(2i t)\bigr).
\end{equation}
(Notice that the left-hand side of \eqref{eq:relation_U_A} is an even function of $t$; so should be the right-hand side, which corresponds to the well-known fact that the $A_{n}$ are reciprocal polynomials.)
We thus have 
\begin{equation*}
        \mathbb P(S_{2n}=0) =\frac{1}{(2n-1)!} \frac{1}{2\pi}\int_{-\pi}^\pi  \exp\bigl(- 2i{(n-1)}t\bigr) A_{2n-1}\bigl(\exp(2i t)\bigr) \dd t=\frac{A(2n-1,n)}{(2n-1)!}.\qedhere
\end{equation*}
\end{proof}
The above proof (mostly \eqref{eq:relation_U_A})\ shows that for $n\geq 1$, 
\begin{equation*}
    \varphi_{2n}(t) = \ee^{-2i(n-1)t} A_{2n-1}(\ee^{2i t}) = \sum_{k=-(n-1)}^{n-1} A(2n-1,k+n)\ee^{2ikt}.
\end{equation*}

\subsection{The case \texorpdfstring{$a=\pm\infty$}{of an infinite memory parameter}}

In this section, we identify the polynomials 
    \begin{equation}
    \label{eq:def_T_n}
        T_n(x)= \lim_{a\to\pm \infty }\frac{R_n(x)}{a^{n-1}},
    \end{equation}
whose existence is guaranteed by Lemma~\ref{lem:obvious_properties_EP} (considered as polynomials in $a$, the $R_n$ have degree $n-1$).
To that purpose, define for any $n\geq 0$ the polynomial $V_n(x)$ such that 
\begin{equation*}
    \tanh^{(n)}(x) = V_n(\tanh x),
\end{equation*}
where $\tanh^{(n)}$ stands for the $n$th derivative of $\tanh$.
For instance, 
\begin{equation*}
    V_0(x)=x,\quad  V_1(x) = 1-x^2 \quad \text{and} \quad V_2(x) =2x^3-2x.
\end{equation*}
See \href{https://oeis.org/A008293}{A008293} and \href{https://oeis.org/A101343}{A101343} for related sequences in the OEIS.

\begin{prop}
\label{prop:a_infty}
For all $n\geq 1$, we have 
\begin{equation}
\label{eq:relation_Tn_Vn}
    T_n(x) = (-1)^{n-1}\frac{V_{n-1}(x)}{(n-1)!}.
\end{equation}
\end{prop}

\begin{proof}
Using the recurrence relation \eqref{eq:recurrence_relation_Rn}, we immediately obtain that the polynomials $T_n(x)$ defined by \eqref{eq:def_T_n} satisfy the recurrence relation
\begin{equation*}
    T_{n+1}(x) = \frac{x^2-1}{n}T_n'(x)
\end{equation*}
for $n\geq 1$, with the initial value $T_1(x) = x$. Define now the polynomials $V_n$ by the relation \eqref{eq:relation_Tn_Vn}. Using the above recurrence relation for the $T_n$, we deduce that the $V_n$ satisfy the recurrence
\begin{equation*}
    V_{n+1}(x) = (1-x^2)V'_{n}(x).
\end{equation*}
This exactly corresponds to our interpretation of the polynomials $V_n$. If indeed $\tanh^{(n)}(x) = V_n(\tanh x)$, then we have
\begin{equation*}
    \tanh^{(n+1)}(x) = (1-\tanh^2 x)V_n'(\tanh x)=V_{n+1}(\tanh x),
\end{equation*}
concluding the proof.
\end{proof}

\section{Connection with the probabilistic behavior of the model}

In this part, our main objective is to prove Proposition~\ref{prop:values_Rn'(1)}, or its generating-function version stated below:
\begin{prop}
\label{prop:derivative_1}
    We have
    \begin{equation}
    \label{eq:derivative_1_series}
        \sum_{n\geq 1} R_n'(1) x^n = \left\{\begin{array}{ll}
        \displaystyle \frac{x}{(1-x)^2}\frac{1-2  a (1-x)^{1-2 a}}{1-2 a} & \text{if } a\neq \frac{1}{2},\medskip\\
        \displaystyle \frac{x}{(1-x)^2}(1-\ln(1 - x)) & \text{if } a=\frac{1}{2}.
        \end{array}\right.
\end{equation}
\end{prop}

\begin{proof}
First note that $R_1'(1)=1$ by \eqref{eq:first_values_polynomials}. Then take the derivative of \eqref{eq:recurrence_relation_Rn} and evaluate the new identity at $x=1$. This way, we obtain for $n\geq 1$
\begin{equation}
\label{eq:rec__Rn'(1)}
   R_{n+1}'(1)=R_n(1)+R_n'(1)+\frac{2a}{n}R_n'(1) = 1+ R_n'(1)\left(1+\frac{2a}{n}\right),
\end{equation}
where we have simplified $R_n(1)=1$ by Lemma~\ref{lem:obvious_properties_EP}. To proceed, we can just check that the sequence 
\begin{equation*}
    \frac{\Gamma(n + 2)\Gamma(2a) - \Gamma( n + 1+2a )}{(1 - 2a)\Gamma(2a)\Gamma(n + 1)},
\end{equation*}
which appears in the right-hand side of \eqref{eq:values_Rn'(1)}, satisfies the same recurrence as \eqref{eq:rec__Rn'(1)} and same initial value for $n=1$. So it should coincide with $R_n'(1)$.

A more constructive proof consists in multiplying \eqref{eq:rec__Rn'(1)} by $n$, and to deduce that the series $y(x) = \sum_{n\geq 1}R_n'(1)x^n$ satisfies the differential equation
\begin{equation*}
   (1 - x)y'(x) - \left(\frac{1}{x} + 2a\right)y(x) = \frac{x}{(1 - x)^2}   
\end{equation*}
with initial condition ($y(0)=0$ and)\ $y'(0)=1$. It is standard to solve this order-one differential equation in closed-form, and to conclude to \eqref{eq:derivative_1_series}. The formula \eqref{eq:values_Rn'(1)} follows from a Taylor expansion of the series \eqref{eq:derivative_1_series}.
\end{proof}

Note that the series $\sum_{n\geq 1} R_n^{(k)}(1) x^n$ corresponding to higher-order derivatives $(k\geq 2)$ admit similar, but considerably more complicated, closed-form expressions.

\paragraph{Acknowledgments.} We would like to thank Alin Bostan for interesting discussions at an early stage of the project, as well as Jean Bertoin and Philippe Nadeau for their valuable comments.

\small
\bibliographystyle{abbrv}

\begin{thebibliography}{10}

\bibitem{BaurBertoin2016}
E.~Baur and J.~Bertoin.
\newblock Elephant random walks and their connection to {P}\'olya-type urns.
\newblock {\em Phys. Rev. E}, 94:052134, 2016.

\bibitem{BercuLaulin2019}
B.~Bercu and L.~Laulin.
\newblock On the multi-dimensional elephant random walk.
\newblock {\em J. Stat. Phys.}, 175(6):1146--1163, 2019.

\bibitem{Bertoinb2021}
J.~Bertoin.
\newblock Counting the zeros of an elephant random walk.
\newblock {\em Trans. Amer. Math. Soc.}, (375):5539--5560, 2022.

\bibitem{Bertoin2023}
J.~Bertoin.
\newblock Counterbalancing steps at random in a random walk.
\newblock {\em J. Eur. Math. Soc.}, (375):1--23, 2023.

\bibitem{Coletti2017}
C.~F. Coletti, R.~Gava, and G.~M. Sch\"{u}tz.
\newblock Central limit theorem and related results for the elephant random
  walk.
\newblock {\em J. Math. Phys.}, 58(5):053303, 8, 2017.

\bibitem{Coletti2019}
C.~F. Coletti and I.~Papageorgiou.
\newblock Asymptotic analysis of the elephant random walk.
\newblock {\em J. Stat. Mech. Theory Exp.}, (1):Paper No. 013205, 12, 2021.

\bibitem{Cressoni2007}
J.~C. Cressoni, M.~A.~A. da~Silva, and G.~M. Viswanathan.
\newblock Amnestically induced persistence in random walks.
\newblock {\em Phys. Rev. Lett.}, 98(7):070603, 4, 2007.

\bibitem{daSilva2013}
M.~A.~A. Da~Silva, J.~C. Cressoni, G.~M. Sch\"utz, G.~M. Viswanathan, and
  S.~Trimper.
\newblock Non-gaussian propagator for elephant random walks.
\newblock {\em Phys. Rev. E}, 88:022115, 2013.

\bibitem{Diaconis-Fulton91}
P.~Diaconis and W.~Fulton.
\newblock A growth model, a game, an algebra, {L}agrange inversion, and
  characteristic classes.
\newblock {\em Rend. Sem. Mat. Univ. Politec. Torino}, 49(1):95--119, 1991.

\bibitem{GLR2023}
H.~Guérin, L.~Laulin, and K.~Raschel.
\newblock A fixed-point equation approach for the superdiffusive elephant
  random walk.
\newblock {\em arXiv}, 2308.14630, 2023.

\bibitem{Kursten2016}
R.~K\"{u}rsten.
\newblock Random recursive trees and the elephant random walk.
\newblock {\em Phys. Rev. E}, 93(3):032111, 11, 2016.

\bibitem{Mahmoud2008}
H.~M. Mahmoud.
\newblock {\em P\'{o}lya urn models}.
\newblock Texts in Statistical Science Series. CRC Press, Boca Raton, FL, 2008.

\bibitem{Mittelstaedt20}
K.~Mittelstaedt.
\newblock A stochastic approach to {E}ulerian numbers.
\newblock {\em Amer. Math. Monthly}, 127(7):618--628, 2020.

\bibitem{qin2023recurrence}
S.~Qin.
\newblock Recurrence and transience of multidimensional elephant random walks.
\newblock {\em arXiv}, 2309.09795, 2023.

\bibitem{Schutz2004}
G.~M. Sch\"utz and S.~Trimper.
\newblock Elephants can always remember: Exact long-range memory effects in a
  non-markovian random walk.
\newblock {\em Phys. Rev. E}, 70:045101, 2004.

\bibitem{oeis}
N.~J.~A. Sloane and T.~O.~F. Inc.
\newblock The on-line encyclopedia of integer sequences, 2020.

\bibitem{zbMATH02510415}
G.~Szeg{\H{o}}.
\newblock {\em Orthogonal polynomials.}, volume~23 of {\em Colloq. Publ., Am.
  Math. Soc.}
\newblock American Mathematical Society (AMS), Providence, RI, 1939.

\bibitem{tanny}
S.~Tanny.
\newblock A probabilistic interpretation of {E}ulerian numbers.
\newblock {\em Duke Math. J.}, 40:717--722, 1973.

\end{thebibliography}

\end{document}